\documentclass[10pt,portrait,reqno,11pt]{amsart}%
\usepackage{amssymb,amsthm,amsmath, amsfonts, float}
\usepackage{graphicx}
\usepackage{multirow}
\usepackage{subfig}
\usepackage{amsmath}
\usepackage{amsfonts}
\usepackage{lscape}
\usepackage{amssymb}%
\providecommand{\U}[1]{\protect\rule{.1in}{.1in}}
\newtheorem{theorem}{Theorem}
\theoremstyle{plain}

\newtheorem{corollary}{Corollary}

\numberwithin{equation}{section}
\begin{document}
\title[Soliton Surf. with the Betchov-Da Rios Eq. using an Extended Darboux Frame
Field in $E^{4}$]{A Geometric Application of Soliton Surfaces associated with the Betchov-Da
Rios Equation using an Extended Darboux Frame Field in $E^{4}$}
\author[A. Kazan, M. Alt\i n]{\bfseries Ahmet Kazan$^{1\ast}$ and Mustafa Alt\i n$^{2}$}
\address{$^{1}$Department of Computer Technologies, Do\u{g}an\c{s}ehir Vahap
K\"{u}\c{c}\"{u}k Vocational School, Malatya Turgut \"{O}zal University,
Malatya, Turkey \\
 \newline
$^{2}$Department of Mathematics, Faculty of Arts and Sciences, Bing\"{o}l University, Bing\"{o}l,
Turkey \\
 \newline
$^{\ast}$Corresponding author: ahmet.kazan@ozal.edu.tr}

\begin{abstract}
In this paper, for a soliton surface $\Omega=\Omega(u,v)$ associated with the
Betchov-Da Rios equation, we obtain the derivative formulas of an extended
Darboux frame field of a unit speed curve $u$-parameter curve $\Omega
=\Omega(u,v)$ for all $v$. Also, we get the geometric invariants $k$ and $h$
of the soliton surface $\Omega=\Omega(u,v)$ and we obtain the Gaussian
curvature, mean curvature vector and Gaussian torsion of $\Omega$. We give
some important geometric characterizations such as flatness, minimality and
semi-umbilicaly with the aid of these invariants. Additionally, we study the
curvature ellipse of the Betchov-Da Rios soliton surface and Wintgen ideal
(superconformal) Betchov-Da Rios soliton surface with respect to an extended
Darboux frame field. Finally, we construct an application for the Betchov-Da
Rios soliton surface with the aid of an extended Darboux frame field.

\end{abstract}
\maketitle


\section{\textbf{General Information and Basic Concepts}}

An important example of integrable curve dynamics is the vortex filament
equation (VFE) which describes the self-induced motion of a vortex filament in
an ideal fluid. The VFE (also known as the smoke ring equation or localized
induction equation (LIE)) is an evolution equation for the space curves in
$R^{3}$ and it was introduced by L.S. Da Rios as a model for the motion of a
one-dimensional vortex filament in an incompressible, inviscid
three-dimensional fluid \cite{Rios}. If the position vector of the vortex
filament is $\Omega=\Omega(u,v)$, then the relation
\[
\Omega_{v}=\Omega_{u}\times\Omega_{uu}%
\]
which is called the vortex filament equation holds. It can also be written in
terms of the Frenet-Serret frame of a space curve $\gamma(u,v)$ as
\[
\gamma_{v}=T\times\kappa N=\kappa B.
\]
Here, the curve $\gamma(u,v)$ is a vector valued function in $R^{3}$; $u$ is
the arc-length parameter; $t$ is the time parameter; $T$, $N$, $B$ are the
tangent, normal, binormal vectors, respectively and $\kappa$ is the curvature
function of the curve $\gamma$.

Furthermore, the thin filament is expressed, smooth and without
self-intersection. The velocity induced by a vortex line at an external point
is expressed by Da Rios via the so-called localized induction approximation
(LIA). The movement of a thin vortex in a thin inviscid fluid by the motion of
a curve propagating in $R^{4}$ is described by the following equation
\begin{equation}
\Omega_{v}=\Omega_{u}\times\Omega_{uu}\times\Omega_{uuu}. \label{BDReq}%
\end{equation}
This is called the Betchov-Da Rios equation or LIE, and can be viewed as a
dynamical system on the space of curves in $R^{4}$. For more details about
vortex filaments and the Betchov-Da Rios equation, we refer to \cite{Barros},
\cite{Betchov}, \cite{Rios}, \cite{Grbovic}, \cite{Hasimoto1},
\cite{Hasimoto2}, \cite{Melek}, and etc.

On the other hand, frame fields are one of the most important tools for
researchers who want to obtain important differential geometric properties of
curves and (hyper)surfaces in three and higher-dimensional spaces. In this
context, Frenet frame fields, one of the most famous of these frame fields,
have been used by researchers to characterize curves in three-dimensional
spaces for a long time. Moreover, the generalization of this frame field to
higher-dimensional spaces is also known and used extensively. In
three-dimensional spaces, another one of the most important alternative frame
fields to the Frenet frame field is the Darboux frame field. Frenet frame
fields have been generalized to higher-dimensional spaces, and many
characterizations of the curves and surfaces in these spaces have been given
in various articles with the help of Frenet frame fields. However, Darboux
frame fields were first moved to four-dimensional space by D\"{u}ld\"{u}l and
his friends in 2017, and this frame field was called the extended Darboux
frame field (\cite{Duldul}). Subsequently, various studies began to be carried
out in four-dimensional spaces regarding this frame field (\cite{Altin},
\cite{Unluturk}, \cite{Baharduldul}, \cite{Baharduldul2}, \cite{Dulduller},
\cite{Kazan}, and etc.).

Now, let us recall the extended Darboux frame field of the second kind along a
curve in four-dimensional Euclidean space $E^{4}$.

We consider an embedding $\Omega:U\subset E^{3}\longrightarrow E^{4}$, where
$U$ is an open subset of $E^{3}$. Now, we denote $M=\Omega(U)$ and identify
$M$ and $U$ through the embedding $\Omega$. Let $\bar{\gamma}:I\longrightarrow
U$ be a regular curve and we have a curve $\gamma:I\longrightarrow M\subset
E^{4}$ defined by $\gamma(u)=$ $\Omega(\bar{\gamma}(u))$ and so, the curve
$\gamma$ is on the hypersurface $M$. If $M$ is an orientable hypersurface
oriented by the unit normal vector field $\mathcal{N}$ in$~E^{4}$ and $\gamma$
is a Frenet curve of class $C^{n}(n\geq4)$ with an arc-length parameter $u$
lying on $M$, then we denote the unit tangent vector field of the curve by $T$
and denote the hypersurface unit normal vector field restricted to the curve
by $N$, i.e.%
\[
T(u)=\gamma^{\prime}(u)\text{\ \ and\ \ }N(u)=\mathcal{N}(\gamma(u)).
\]
If the set $\{N,T,\gamma^{\prime\prime}\}$ is linearly dependent, from the
Gram-Schmidt orthonormalization method, $\{N,T,\gamma^{\prime\prime\prime}\}$
yields the orthonormal set $\{N,T,E\}$, where%
\[
E=\frac{\gamma^{\prime\prime\prime}-\left\langle \gamma^{\prime\prime\prime
},N\right\rangle N-\left\langle \gamma^{\prime\prime\prime},T\right\rangle
T}{\left\Vert \gamma^{\prime\prime\prime}-\left\langle \gamma^{\prime
\prime\prime},N\right\rangle N-\left\langle \gamma^{\prime\prime\prime
},T\right\rangle T\right\Vert }.
\]
Defining $D=N\times T\times E$, we have a new orthonormal frame field
$\left\{  T,E,D,N\right\}  $ along the curve $\gamma$ and for simplicity,
we'll call it \textit{ED}$^{2}$\textit{-frame field}. The differential
equations of ED$^{2}$-frame fields $\left\{  T,E,D,N\right\}  $ of the curve
$\gamma$ in $E^{4}$\ can be given as%
\begin{equation}
\left.
\begin{array}
[c]{l}%
T^{\prime}=\kappa_{n}N,\\
E^{\prime}=\kappa_{g}^{2}D+\tau_{g}^{1}N,\\
D^{\prime}=-\kappa_{g}^{2}E,\\
N^{\prime}=-\kappa_{n}T-\tau_{g}^{1}E,
\end{array}
\right\}  \label{darboux}%
\end{equation}
where $\kappa_{n}=\left\langle T^{\prime},N\right\rangle $ is the normal
curvature of the hypersurface in the direction of the tangent vector $T,$
$\kappa_{g}^{2}$ $=\left\langle E^{\prime},D\right\rangle $ is the geodesic
curvature of order $2$ and $\tau_{g}^{1}=\left\langle E^{\prime}%
,N\right\rangle $ is the geodesic torsion of order $1$. For more details about
the construction of the extended Darboux frame fields, we refer to
\cite{Duldul}.

\section{\textbf{Betchov-Da Rios soliton equation with respect to the ED}%
$^{2}$\textbf{-frame field in }$E^{4}$}

In this section, for a soliton surface $\Omega=\Omega(u,v)$ associated with
the Betchov-Da Rios equation, we will obtain the derivative formulas of an
extended Darboux frame field of a unit speed curve $u$-parameter curve
$\Omega=\Omega(u,v)$ for all $v$. Throughout this study, it is important to
note that we assume the normal curvature ($\kappa_{n}$) and first order
geodesic torsion ($\tau_{g}^{1}$) are non-zero.

Firstly, if $\Omega=\Omega(u,v)$ is a solution of the Betchov-Da Rios equation
such that the $u$-parameter curve $\Omega=\Omega(u,v)$ is a unit speed curve
for all $v$, then from (\ref{darboux}) we get the derivative formulas of the
ED$^{2}$-frame field according to "$u$" as%
\begin{equation}
\left.
\begin{array}
[c]{l}%
T_{u}(u,v)=\kappa_{n}(u,v)N(u,v),\\
E_{u}(u,v)=\kappa_{g}^{2}(u,v)D(u,v)+\tau_{g}^{1}(u,v)N(u,v),\\
D_{u}(u,v)=-\kappa_{g}^{2}(u,v)E(u,v),\\
N_{u}(u,v)=-\kappa_{n}(u,v)T(u,v)-\tau_{g}^{1}(u,v)E(u,v).
\end{array}
\right\}  \label{Fiu}%
\end{equation}

Now, let us obtain the derivative formulas according to "$v$". For this, we
must find the smooth functions $a_{ij},$ $i,j\in\{1,2,3,4\}$ of the equations%
\[
\left.
\begin{array}
[c]{c}%
T_{v}(u,v)=a_{11}(u,v)T(u,v)+a_{12}(u,v)E(u,v)+a_{13}(u,v)D(u,v)+a_{14}%
(u,v)N(u,v),\\
E_{v}(u,v)=a_{21}(u,v)T(u,v)+a_{22}(u,v)E(u,v)+a_{23}(u,v)D(u,v)+a_{24}%
(u,v)N(u,v),\\
D_{v}(u,v)=a_{31}(u,v)T(u,v)+a_{32}(u,v)E(u,v)+a_{33}(u,v)D(u,v)+a_{34}%
(u,v)N(u,v),\\
N_{v}(u,v)=a_{41}(u,v)T(u,v)+a_{42}(u,v)E(u,v)+a_{43}(u,v)D(u,v)+a_{44}%
(u,v)N(u,v).
\end{array}
\right\}
\]

From $\left\langle T,T\right\rangle =\left\langle E,E\right\rangle
=\left\langle D,D\right\rangle =\left\langle N,N\right\rangle =1$ and
$\left\langle T,E\right\rangle =\left\langle T,D\right\rangle =\left\langle
T,N\right\rangle =\left\langle E,D\right\rangle =\left\langle E,N\right\rangle
=\left\langle D,N\right\rangle =0$, we have $a_{ii}(u,v)=0$ and $a_{ij}%
(u,v)=-a_{ji}(u,v)$ ($i\neq j$) and so we can write%
\begin{equation}
\left.
\begin{array}
[c]{l}%
T_{v}(u,v)=a_{12}(u,v)E(u,v)+a_{13}(u,v)D(u,v)+a_{14}(u,v)N(u,v),\\
E_{v}(u,v)=-a_{12}(u,v)T(u,v)+a_{23}(u,v)D(u,v)+a_{24}(u,v)N(u,v),\\
D_{v}(u,v)=-a_{13}(u,v)T(u,v)-a_{23}(u,v)E(u,v)+a_{34}(u,v)N(u,v),\\
N_{v}(u,v)=-a_{14}(u,v)T(u,v)-a_{24}(u,v)E(u,v)-a_{34}(u,v)D(u,v).
\end{array}
\right\}  \label{Fiv}%
\end{equation}
Here we must note that we will not write $(u,v)$ for simplicity in
$a_{ij}(u,v)$, $T(u,v)$, and so on. Also, we will use the notation
$\frac{\partial f}{\partial u}$ and $f_{u}$ interchangeably, and similarly
with higher order derivatives; i.e. $\frac{\partial^{2}f}{\partial u\partial
v}$ is the same as $f_{uv}$, and so on.

Let us find the functions $a_{12}$, $a_{13}$, $a_{14}$, $a_{23}$,
$a_{24}$ and $a_{34}$. Using
\begin{equation}
\Omega_{u}=T \label{omegau}%
\end{equation}
and (\ref{Fiu}), we have%
\begin{equation}
\Omega_{uu}=\kappa_{n}N \label{omegauu}%
\end{equation}
and%
\begin{equation}
\Omega_{uuu}=-\kappa_{n}^{2}T-\kappa_{n}\tau_{g}^{1}E+\left(  \kappa
_{n}\right)  _{u}N. \label{omegauuu}%
\end{equation}
From (\ref{omegau})-(\ref{omegauuu}) and the Betchov-Da Rios equation
(\ref{BDReq}), we reach that%
\begin{equation}
\Omega_{v}=\left(  \kappa_{n}\right)  ^{2}\tau_{g}^{1}D. \label{omegav}%
\end{equation}

On the other hand, from (\ref{Fiv}) and (\ref{omegau}) we have%
\begin{equation}
\Omega_{uv}=a_{12}E+a_{13}D+a_{14}N \label{omegauv}%
\end{equation}
and from (\ref{Fiu}) and (\ref{omegav}) we get%
\begin{equation}
\Omega_{vu}=\left(  -\left(  \kappa_{n}\right)  ^{2}\tau_{g}^{1}\kappa_{g}%
^{2}\right)  E+\left(  \kappa_{n}\left(  2\tau_{g}^{1}\left(  \kappa
_{n}\right)  _{u}+\kappa_{n}\left(  \tau_{g}^{1}\right)  _{u}\right)  \right)
D. \label{omegavu}%
\end{equation}
We know that we have the compatibility condition $f_{uv}=f_{vu}$ for a $C^{2}%
$-function $f$. Thus from $\Omega_{uv}=\Omega_{vu}$, (\ref{omegauv}) and
(\ref{omegavu}), we get%
\begin{align}
a_{12}  &  =-\left(  \kappa_{n}\right)  ^{2}\tau_{g}^{1}\kappa_{g}%
^{2},\label{a12}\\
a_{13}  &  =\kappa_{n}\left(  2\tau_{g}^{1}\left(  \kappa_{n}\right)
_{u}+\kappa_{n}\left(  \tau_{g}^{1}\right)  _{u}\right)  ,\label{a13}\\
a_{14}  &  =0. \label{a14}%
\end{align}
Now, let us give $T_{uv},$ $T_{vu}$ and the equations obtained by
$T_{uv}=T_{vu}$, and so on.

Using $T_{uv}=T_{vu}$,%
\begin{equation}
T_{uv}=\left(  -a_{14}\kappa_{n}\right)  T+\left(  -a_{24}\kappa_{n}\right)
E+\left(  -a_{34}\kappa_{n}\right)  D+\left(  \left(  \kappa_{n}\right)
_{v}\right)  N \label{F1uv}%
\end{equation}
and%
\begin{equation}
T_{vu}=\left(  -a_{14}\kappa_{n}\right)  T+\left(  \left(  a_{12}\right)
_{u}-a_{13}\kappa_{g}^{2}-a_{14}\tau_{g}^{1}\right)  E+\left(  a_{12}%
\kappa_{g}^{2}+\left(  a_{13}\right)  _{u}\right)  D+\left(  a_{12}\tau
_{g}^{1}+\left(  a_{14}\right)  _{u}\right)  N, \label{F1vu}%
\end{equation}
we have%
\begin{align}
&  -a_{24}\kappa_{n}=\left(  a_{12}\right)  _{u}-a_{13}\kappa_{g}^{2}%
-a_{14}\tau_{g}^{1},\label{F1uv2}\\
&  -a_{34}\kappa_{n}=a_{12}\kappa_{g}^{2}+\left(  a_{13}\right)
_{u},\label{F1uv3}\\
&  \left(  \kappa_{n}\right)  _{v}=a_{12}\tau_{g}^{1}+\left(  a_{14}\right)
_{u}. \label{F1uv4}%
\end{align}
Using (\ref{a12})-(\ref{a14}) in (\ref{F1uv2}) and (\ref{F1uv3}), we get%
\begin{equation}
a_{24}=2\kappa_{n}\kappa_{g}^{2}\left(  \tau_{g}^{1}\right)  _{u}+\tau_{g}%
^{1}\left(  \kappa_{n}\left(  \kappa_{g}^{2}\right)  _{u}+4\kappa_{g}%
^{2}\left(  \kappa_{n}\right)  _{u}\right)  \label{a24}%
\end{equation}
and%
\begin{equation}
a_{34}=\frac{1}{\kappa_{n}}\left(  \left(  \kappa_{n}\tau_{g}^{1}(\kappa
_{g}^{2})^{2}-4\left(  \kappa_{n}\right)  _{u}\left(  \tau_{g}^{1}\right)
_{u}-\kappa_{n}\left(  \tau_{g}^{1}\right)  _{uu}\right)  \kappa_{n}-2\left(
\left(  \left(  \kappa_{n}\right)  _{u}\right)  ^{2}+\kappa_{n}\left(
\kappa_{n}\right)  _{uu}\right)  \tau_{g}^{1}\right)  , \label{a34}%
\end{equation}
respectively.

If we use $E_{uv}=E_{vu}$,%
\begin{equation}
E_{uv}=\left(  -a_{14}\tau_{g}^{1}-a_{13}\kappa_{g}^{2}\right)  T+\left(
-a_{23}\kappa_{g}^{2}-a_{24}\tau_{g}^{1}\right)  E+\left(  \left(  \kappa
_{g}^{2}\right)  _{v}-a_{34}\tau_{g}^{1}\right)  D+\left(  a_{34}\kappa
_{g}^{2}+\left(  \tau_{g}^{1}\right)  _{v}\right)  N \label{F2uv}%
\end{equation}
and%
\begin{equation}
E_{vu}=\left(  -\left(  a_{12}\right)  _{u}-a_{24}\kappa_{n}\right)  T+\left(
-a_{23}\kappa_{g}^{2}-a_{24}\tau_{g}^{1}\right)  E+\left(  \left(
a_{23}\right)  _{u}\right)  D+\left(  -a_{12}\kappa_{n}+\left(  a_{24}\right)
_{u}\right)  N, \label{F2vu}%
\end{equation}
then we have%
\begin{align}
&  a_{14}\tau_{g}^{1}+a_{13}\kappa_{g}^{2}=\left(  a_{12}\right)  _{u}%
+a_{24}\kappa_{n},\label{F2uv1}\\
&  \left(  \kappa_{g}^{2}\right)  _{v}-a_{34}\tau_{g}^{1}=\left(
a_{23}\right)  _{u},\label{F2uv3}\\
&  a_{34}\kappa_{g}^{2}+\left(  \tau_{g}^{1}\right)  _{v}=-a_{12}\kappa
_{n}+\left(  a_{24}\right)  _{u}. \label{F2uv4}%
\end{align}
From $D_{uv}=D_{vu}$,%
\begin{equation}
D_{uv}=\left(  a_{12}\kappa_{g}^{2}\right)  T+\left(  -\left(  \kappa_{g}%
^{2}\right)  _{v}\right)  E+\left(  -a_{23}\kappa_{g}^{2}\right)  D+\left(
-a_{24}\kappa_{g}^{2}\right)  N \label{F3uv}%
\end{equation}
and%
\begin{equation}
D_{vu}=\left(  -\left(  a_{13}\right)  _{u}-a_{34}\kappa_{n}\right)  T+\left(
-\left(  a_{23}\right)  _{u}-a_{34}\tau_{g}^{1}\right)  E+\left(
-a_{23}\kappa_{g}^{2}\right)  D+\left(  -a_{13}\kappa_{n}-a_{23}\tau_{g}%
^{1}+\left(  a_{34}\right)  _{u}\right)  N, \label{F3vu}%
\end{equation}
we get%
\begin{align}
&  a_{12}\kappa_{g}^{2}=-\left(  a_{13}\right)  _{u}-a_{34}\kappa
_{n},\label{F3uv1}\\
&  \left(  \kappa_{g}^{2}\right)  _{v}=\left(  a_{23}\right)  _{u}+a_{34}%
\tau_{g}^{1},\label{F3uv2}\\
&  a_{24}\kappa_{g}^{2}=a_{13}\kappa_{n}+a_{23}\tau_{g}^{1}-\left(
a_{34}\right)  _{u}. \label{F3uv4}%
\end{align}
Using (\ref{a13}), (\ref{a24}) and (\ref{a34}) in (\ref{F3uv4}), we reach that%
\begin{equation}
a_{23}=\frac{-1}{\tau_{g}^{1}(\kappa_{n})^{2}}\left(
\begin{array}
[c]{l}%
-3(\kappa_{n})^{3}\tau_{g}^{1}\kappa_{g}^{2}\left(  \kappa_{g}^{2}\right)
_{u}+(\kappa_{n})^{5}\left(  \tau_{g}^{1}\right)  _{u}+2(\kappa_{n})^{4}%
\tau_{g}^{1}\left(  \kappa_{n}\right)  _{u}-2\tau_{g}^{1}\left(  \left(
\kappa_{n}\right)  _{u}\right)  ^{3}\\
-\left(  \kappa_{n}\kappa_{g}^{2}\right)  ^{2}\left(  3\kappa_{n}\left(
\tau_{g}^{1}\right)  _{u}+5\tau_{g}^{1}\left(  \kappa_{n}\right)  _{u}\right)
+2\kappa_{n}\left(  \kappa_{n}\right)  _{u}\left(  \left(  \kappa_{n}\right)
_{u}\left(  \tau_{g}^{1}\right)  _{u}+2\tau_{g}^{1}\left(  \kappa_{n}\right)
_{uu}\right) \\
+(\kappa_{n})^{3}\left(  \tau_{g}^{1}\right)  _{uuu}+(\kappa_{n})^{2}\left(
5\left(  \kappa_{n}\right)  _{u}\left(  \tau_{g}^{1}\right)  _{uu}+6\left(
\tau_{g}^{1}\right)  _{u}\left(  \kappa_{n}\right)  _{uu}+2\tau_{g}^{1}\left(
\kappa_{n}\right)  _{uuu}\right)
\end{array}
\right)  . \label{a23}%
\end{equation}
If we use%
\begin{equation}
N_{uv}=\left(  -\left(  \kappa_{n}\right)  _{v}+a_{12}\tau_{g}^{1}\right)
T+\left(  -a_{12}\kappa_{n}-\left(  \tau_{g}^{1}\right)  _{v}\right)
E+\left(  -a_{13}\kappa_{n}-a_{23}\tau_{g}^{1}\right)  D+\left(  -a_{14}%
\kappa_{n}-a_{24}\tau_{g}^{1}\right)  N \label{F4uv}%
\end{equation}
and%
\begin{equation}
N_{vu}=\left(  -\left(  a_{14}\right)  _{u}\right)  T+\left(  -\left(
a_{24}\right)  _{u}+a_{34}\kappa_{g}^{2}\right)  E+\left(  -\left(
a_{34}\right)  _{u}-a_{24}\kappa_{g}^{2}\right)  D+\left(  -a_{14}\kappa
_{n}-a_{24}\tau_{g}^{1}\right)  N \label{F4vu}%
\end{equation}
in $N_{uv}=N_{vu}$, then we get the equations (\ref{F1uv4}), (\ref{F2uv4}) and
(\ref{F3uv4}), again.

Hence, after the above calculations, we can give the following results:

\begin{theorem}
If the $u$-parameter curve $\Omega=\Omega(u,v)$ is unit speed for all $v$ and
$\Omega=\Omega(u,v)$ is a solution of the Betchov-Da Rios equation with
respect to the ED$^{2}$-frame field in $E^{4}$, then the derivative formulas
of the ED$^{2}$-frame field are%
\[
\left[
\begin{array}
[c]{c}%
T_{u}\\
E_{u}\\
D_{u}\\
N_{u}%
\end{array}
\right]  =\left[
\begin{array}
[c]{cccc}%
0 & 0 & 0 & \kappa_{n}\\
0 & 0 & \kappa_{g}^{2} & \tau_{g}^{1}\\
0 & -\kappa_{g}^{2} & 0 & 0\\
-\kappa_{n} & -\tau_{g}^{1} & 0 & 0
\end{array}
\right]  \left[
\begin{array}
[c]{c}%
T\\
E\\
D\\
N
\end{array}
\right]
\]
and%
\[
\left[
\begin{array}
[c]{c}%
T_{v}\\
E_{v}\\
D_{v}\\
N_{v}%
\end{array}
\right]  =\left[
\begin{array}
[c]{cccc}%
0 & a_{12} & a_{13} & a_{14}\\
-a_{12} & 0 & a_{23} & a_{24}\\
-a_{13} & -a_{23} & 0 & a_{34}\\
-a_{14} & -a_{24} & -a_{34} & 0
\end{array}
\right]  \left[
\begin{array}
[c]{c}%
T\\
E\\
D\\
N
\end{array}
\right]  ,
\]
where $a_{12},$ $a_{13},$ $a_{14},$ $a_{23},$ $a_{24}$ and $a_{34}$ are given
by (\ref{a12}), (\ref{a13}), (\ref{a14}), (\ref{a23}), (\ref{a24}) and
(\ref{a34}), respectively.
\end{theorem}

\begin{corollary}
If the $u$-parameter curve $\Omega=\Omega(u,v)$ is unit speed for all $v$ and
$\Omega=\Omega(u,v)$ is a solution of the Betchov-Da Rios equation with
respect to the ED$^{2}$-frame field in $E^{4}$, then we get the following
equations:%
\begin{equation}
(\kappa_{n}\tau_{g}^{1})^{2}\kappa_{g}^{2}+\left(  \kappa_{n}\right)  _{v}=0,
\label{sari1}%
\end{equation}%
\begin{align}
&  \left(  \kappa_{n}\tau_{g}^{1}(\kappa_{g}^{2})^{2}-10\left(  \kappa
_{n}\right)  _{u}\left(  \tau_{g}^{1}\right)  _{u}-3\kappa_{n}\left(  \tau
_{g}^{1}\right)  _{uu}\right)  \kappa_{n}\kappa_{g}^{2}-\left(  2(\left(
\kappa_{n}\right)  _{u})^{2}+\kappa_{n}((\kappa_{n})^{3}+6\left(  \kappa
_{n}\right)  _{uu})\right)  \tau_{g}^{1}\kappa_{g}^{2}\nonumber\\
\text{ }  &  +\left(  \left(  \tau_{g}^{1}\right)  _{v}-\left(  3\kappa
_{n}\left(  \tau_{g}^{1}\right)  _{u}+5\tau_{g}^{1}\left(  \kappa_{n}\right)
_{u}\right)  \left(  \kappa_{g}^{2}\right)  _{u}-\kappa_{n}\tau_{g}^{1}\left(
\kappa_{g}^{2}\right)  _{uu}\right)  \kappa_{n}=0 \label{sari3}%
\end{align}
and%
\begin{align}
&  (\kappa_{n})^{3}\tau_{g}^{1}\kappa_{g}^{2}\left(  \left(  \kappa_{g}%
^{2}\right)  _{u}\left(  6\kappa_{n}\left(  \tau_{g}^{1}\right)  _{u}%
+13\tau_{g}^{1}\left(  \kappa_{n}\right)  _{u}\right)  +3\kappa_{n}\tau
_{g}^{1}\left(  \kappa_{g}^{2}\right)  _{uu}\right) \nonumber\\
&  -(\kappa_{n}\tau_{g}^{1})^{3}\left(  4\left(  \kappa_{n}\right)
_{u}\left(  \tau_{g}^{1}\right)  _{u}+\kappa_{n}\left(  \tau_{g}^{1}\right)
_{uu}\right)  -2(\kappa_{n})^{2}(\tau_{g}^{1})^{4}\left(  \left(  \left(
\kappa_{n}\right)  _{u}\right)  ^{2}+\kappa_{n}\left(  \kappa_{n}\right)
_{uu}\right) \nonumber\\
&  +(\kappa_{n})^{3}(\kappa_{g}^{2})^{2}\left(  \kappa_{n}(\tau_{g}^{1}%
)^{4}-3\kappa_{n}\left(  \left(  \tau_{g}^{1}\right)  _{u}\right)
^{2}+3\left(  \left(  \kappa_{n}\right)  _{u}\left(  \tau_{g}^{1}\right)
_{u}+\kappa_{n}\left(  \tau_{g}^{1}\right)  _{uu}\right)  \tau_{g}%
^{1}+5\left(  \tau_{g}^{1}\right)  ^{2}\left(  \kappa_{n}\right)  _{uu}\right)
\nonumber\\
&  +(\kappa_{n})^{2}\left(  \tau_{g}^{1}\right)  _{u}\left(  (\kappa_{n}%
)^{4}\left(  \tau_{g}^{1}\right)  _{u}+2\left(  \left(  \kappa_{n}\right)
_{u}\right)  ^{2}\left(  \tau_{g}^{1}\right)  _{u}+\kappa_{n}\left(  5\left(
\kappa_{n}\right)  _{u}\left(  \tau_{g}^{1}\right)  _{uu}+6\left(  \tau
_{g}^{1}\right)  _{u}\left(  \kappa_{n}\right)  _{uu}\right)  +(\kappa
_{n})^{2}\left(  \tau_{g}^{1}\right)  _{uuu}\right) \nonumber\\
&  -\kappa_{n}\tau_{g}^{1}\left(
\begin{array}
[c]{l}%
3(\kappa_{n})^{4}\left(  \kappa_{n}\right)  _{u}\left(  \tau_{g}^{1}\right)
_{u}-2\left(  \left(  \kappa_{n}\right)  _{u}\right)  ^{3}\left(  \tau_{g}%
^{1}\right)  _{u}+2\kappa_{n}\left(  \kappa_{n}\right)  _{u}\left(  \left(
\kappa_{n}\right)  _{u}\left(  \tau_{g}^{1}\right)  _{uu}+2\left(  \tau
_{g}^{1}\right)  _{u}\left(  \kappa_{n}\right)  _{uu}\right) \\
+(\kappa_{n})^{2}\left(  11\left(  \kappa_{n}\right)  _{uu}\left(  \tau
_{g}^{1}\right)  _{uu}+6\left(  \kappa_{n}\right)  _{u}\left(  \tau_{g}%
^{1}\right)  _{uuu}+6\left(  \tau_{g}^{1}\right)  _{u}\left(  \kappa
_{n}\right)  _{uuu}\right)  +(\kappa_{n})^{5}\left(  \tau_{g}^{1}\right)
_{uu}+(\kappa_{n})^{3}\left(  \tau_{g}^{1}\right)  _{uuuu}%
\end{array}
\right) \nonumber\\
&  \text{ }-(\tau_{g}^{1})^{2}\left(
\begin{array}
[c]{l}%
4\left(  \left(  \kappa_{n}\right)  _{u}\right)  ^{4}+(\kappa_{n})^{4}\left(
4\left(  \left(  \kappa_{n}\right)  _{u}\right)  ^{2}-3\left(  \left(
\kappa_{g}^{2}\right)  _{u}\right)  ^{2}\right)  +2(\kappa_{n})^{5}\left(
\kappa_{n}\right)  _{uu}-10\kappa_{n}\left(  \left(  \kappa_{n}\right)
_{u}\right)  ^{2}\left(  \kappa_{n}\right)  _{uu}\\
+4(\kappa_{n})^{2}\left(  \left(  \left(  \kappa_{n}\right)  _{uu}\right)
^{2}+\left(  \kappa_{n}\right)  _{u}\left(  \kappa_{n}\right)  _{uuu}\right)
+(\kappa_{n})^{3}\left(  \left(  \kappa_{g}^{2}\right)  _{v}+2\left(
\kappa_{n}\right)  _{uuuu}\right)
\end{array}
\right)  =0. \label{sari2}%
\end{align}

\end{corollary}

\begin{proof}
Firstly, from (\ref{a14}) and (\ref{F1uv4}), we have $a_{12}=\frac{\left(
\kappa_{n}\right)  _{v}}{\tau_{g}^{1}}$. Using (\ref{a12}) in the last
equation, we get (\ref{sari1}).

Now, from (\ref{a12}), (\ref{a24}) and (\ref{F2uv4}), we have an alternative
equation for $a_{34}$ as%
\begin{equation}
a_{34}=\frac{1}{\kappa_{g}^{2}}\left(
\begin{array}
[c]{l}%
-\left(  \tau_{g}^{1}\right)  _{v}+\left(  3\kappa_{n}\left(  \tau_{g}%
^{1}\right)  _{u}+5\tau_{g}^{1}\left(  \kappa_{n}\right)  _{u}\right)  \left(
\kappa_{g}^{2}\right)  _{u}+\kappa_{n}\tau_{g}^{1}\left(  \kappa_{g}%
^{2}\right)  _{uu}\\
+\left(  6\left(  \kappa_{n}\right)  _{u}\left(  \tau_{g}^{1}\right)
_{u}+2\kappa_{n}\left(  \tau_{g}^{1}\right)  _{uu}+\left(  (\kappa_{n}%
)^{3}+4\left(  \kappa_{n}\right)  _{uu}\right)  \tau_{g}^{1}\right)
\kappa_{g}^{2}%
\end{array}
\right)  ,\label{a34y}%
\end{equation}
where $\kappa_{g}^{2}\neq0$. Thus, (\ref{sari3}) can be obtained by equalizing
(\ref{a34}) and (\ref{a34y}).

Finally, using (\ref{a34}) and (\ref{a23}) in (\ref{F2uv3}), we obtain
(\ref{sari2}).
\end{proof}

\section{\textbf{Geometric Characterizations for Betchov-Da Rios soliton
surface with respect to the ED}$^{2}$\textbf{-frame field in }$E^{4}$}

In this section, we obtain two invariants $k$ and $h$ introduced in
\cite{Ganchev} of a two-dimensional Betchov-Da Rios soliton surface
$S:\Omega=\Omega(u,v)$ with respect to the ED$^{2}$-frame field in $E^{4}$.
Additionally, we provide some characterizations for this surface by obtaining
its Gaussian curvature, mean curvature vector field and Gaussian torsion.

Firstly, we obtain the coefficients of the first fundamental form as%

\begin{equation}
\left.
\begin{array}
[c]{l}%
g_{11}=\left\langle \Omega_{u},\Omega_{u}\right\rangle =1,\\
g_{12}=g_{21}=\left\langle \Omega_{u},\Omega_{v}\right\rangle =0,\\
g_{22}=\left\langle \Omega_{v},\Omega_{v}\right\rangle =(\kappa_{n}%
)^{4}\left(  \tau_{g}^{1}\right)  ^{2}%
\end{array}
\right\}  \label{gij}%
\end{equation}
and from (\ref{gij}), let us set%
\begin{equation}
\mathcal{W=}\sqrt{g_{11}g_{22}-(g_{12})^{2}}=(\kappa_{n})^{2}\tau_{g}^{1}.
\label{W}%
\end{equation}

If $\Gamma_{ij}^{k}$ $(i,j,k=1,2)$ are the Christoffel's symbols and
$c_{ij}^{k}$ are functions on $S$, then we have the following standard
derivative formulas for the orthonormal normal frame field $\{E,N\}$ of $S$:%
\begin{equation}
\left.
\begin{array}
[c]{l}%
\Omega_{uu}=\Gamma_{11}^{1}\Omega_{u}+\Gamma_{11}^{2}\Omega_{v}+c_{11}%
^{1}E+c_{11}^{2}N,\\
\Omega_{uv}=\Gamma_{12}^{1}\Omega_{u}+\Gamma_{12}^{2}\Omega_{v}+c_{12}%
^{1}E+c_{12}^{2}N,\\
\Omega_{vv}=\Gamma_{22}^{1}\Omega_{u}+\Gamma_{22}^{2}\Omega_{v}+c_{22}%
^{1}E+c_{22}^{2}N.
\end{array}
\right\}  \label{omegaij}%
\end{equation}
On the other hand, from (\ref{Fiv}), (\ref{omegav}), (\ref{a13}), (\ref{a34})
and (\ref{a23}) we have%
\begin{align}
&  \Omega_{vv}=\left(  -(\kappa_{n})^{3}\tau_{g}^{1}\left(  \kappa_{n}\left(
\tau_{g}^{1}\right)  _{u}+2\tau_{g}^{1}\left(  \kappa_{n}\right)  _{u}\right)
\right)  T\nonumber\\
&  +\left(  (\kappa_{n})^{2}\tau_{g}^{1}\left(
\begin{array}
[c]{l}%
-3\kappa_{n}\kappa_{g}^{2}\left(  \kappa_{g}^{2}\right)  _{u}-\frac{\left(
\kappa_{g}^{2}\right)  ^{2}}{\tau_{g}^{1}}\left(  3\kappa_{n}\left(  \tau
_{g}^{1}\right)  _{u}+5\tau_{g}^{1}\left(  \kappa_{n}\right)  _{u}\right)
+2\left(  \kappa_{n}\right)  _{uuu}\\
+\frac{1}{(\kappa_{n})^{2}\tau_{g}^{1}}\left(
\begin{array}
[c]{l}%
(\kappa_{n})^{5}\left(  \tau_{g}^{1}\right)  _{u}+2(\kappa_{n})^{4}\tau
_{g}^{1}\left(  \kappa_{n}\right)  _{u}-2\tau_{g}^{1}(\left(  \kappa
_{n}\right)  _{u})^{3}\\
+2\kappa_{n}\left(  \kappa_{n}\right)  _{u}\left(  \left(  \kappa_{n}\right)
_{u}\left(  \tau_{g}^{1}\right)  _{u}+2\tau_{g}^{1}\left(  \kappa_{n}\right)
_{uu}\right)  +(\kappa_{n})^{3}\left(  \tau_{g}^{1}\right)  _{uuu}\\
+(\kappa_{n})^{2}\left(  5\left(  \kappa_{n}\right)  _{u}\left(  \tau_{g}%
^{1}\right)  _{uu}+6\left(  \tau_{g}^{1}\right)  _{u}\left(  \kappa
_{n}\right)  _{uu}\right)
\end{array}
\right)
\end{array}
\right)  \right)  E\nonumber\\
&  +\left(  \kappa_{n}\left(  \kappa_{n}\left(  \tau_{g}^{1}\right)
_{v}+2\tau_{g}^{1}\left(  \kappa_{n}\right)  _{v}\right)  \right)
D\nonumber\\
&  +\left(  (\kappa_{n})^{2}\tau_{g}^{1}\left(  \kappa_{n}\tau_{g}^{1}\left(
\kappa_{g}^{2}\right)  ^{2}-4\left(  \kappa_{n}\right)  _{u}\left(  \tau
_{g}^{1}\right)  _{u}-\kappa_{n}\left(  \tau_{g}^{1}\right)  _{uu}-\frac
{2\tau_{g}^{1}}{\kappa_{n}}\left(  (\left(  \kappa_{n}\right)  _{u}%
)^{2}+\kappa_{n}\left(  \kappa_{n}\right)  _{uu}\right)  \right)  \right)
N.\label{omegavvy}%
\end{align}
Thus, from (\ref{omegauu}), (\ref{omegavu}) and (\ref{omegavvy}), we get%
\begin{equation}
\left.
\begin{array}
[c]{l}%
c_{11}^{1}=\left\langle \Omega_{uu},E\right\rangle =0,\\
c_{11}^{2}=\left\langle \Omega_{uu},N\right\rangle =\kappa_{n},\\
c_{12}^{1}=\left\langle \Omega_{uv},E\right\rangle =-(\kappa_{n})^{2}\tau
_{g}^{1}\kappa_{g}^{2},\\
c_{12}^{2}=\left\langle \Omega_{uv},N\right\rangle =0,\\
c_{22}^{1}=\left\langle \Omega_{vv},E\right\rangle =(\kappa_{n})^{2}\tau
_{g}^{1}\left(
\begin{array}
[c]{c}%
-3\kappa_{n}\kappa_{g}^{2}\left(  \kappa_{g}^{2}\right)  _{u}-\frac{\left(
\kappa_{g}^{2}\right)  ^{2}}{\tau_{g}^{1}}\left(  3\kappa_{n}\left(  \tau
_{g}^{1}\right)  _{u}+5\tau_{g}^{1}\left(  \kappa_{n}\right)  _{u}\right)
+2\left(  \kappa_{n}\right)  _{uuu}\\
+\frac{1}{(\kappa_{n})^{2}\tau_{g}^{1}}\left(
\begin{array}
[c]{l}%
(\kappa_{n})^{5}\left(  \tau_{g}^{1}\right)  _{u}+2(\kappa_{n})^{4}\tau
_{g}^{1}\left(  \kappa_{n}\right)  _{u}-2\tau_{g}^{1}(\left(  \kappa
_{n}\right)  _{u})^{3}\\
+2\kappa_{n}\left(  \kappa_{n}\right)  _{u}\left(  \left(  \kappa_{n}\right)
_{u}\left(  \tau_{g}^{1}\right)  _{u}+2\tau_{g}^{1}\left(  \kappa_{n}\right)
_{uu}\right)  \\
+(\kappa_{n})^{2}\left(  5\left(  \kappa_{n}\right)  _{u}\left(  \tau_{g}%
^{1}\right)  _{uu}+6\left(  \tau_{g}^{1}\right)  _{u}\left(  \kappa
_{n}\right)  _{uu}\right)  \\
+(\kappa_{n})^{3}\left(  \tau_{g}^{1}\right)  _{uuu}%
\end{array}
\right)
\end{array}
\right)  ,\\
c_{22}^{2}=\left\langle \Omega_{vv},N\right\rangle =(\kappa_{n})^{2}\tau
_{g}^{1}\left(
\begin{array}
[c]{l}%
\kappa_{n}\tau_{g}^{1}\left(  \kappa_{g}^{2}\right)  ^{2}-4\left(  \kappa
_{n}\right)  _{u}\left(  \tau_{g}^{1}\right)  _{u}-\kappa_{n}\left(  \tau
_{g}^{1}\right)  _{uu}\\
-\frac{2\tau_{g}^{1}}{\kappa_{n}}\left(  (\left(  \kappa_{n}\right)  _{u}%
)^{2}+\kappa_{n}\left(  \kappa_{n}\right)  _{uu}\right)
\end{array}
\right)  .
\end{array}
\right\}  \label{cijk}%
\end{equation}
If we introduce the following functions%
\begin{equation}
\left.
\begin{array}
[c]{l}%
\Delta_{1}=\left\vert
\begin{array}
[c]{cc}%
c_{11}^{1} & c_{12}^{1}\\
c_{11}^{2} & c_{12}^{2}%
\end{array}
\right\vert =(\kappa_{n})^{3}\tau_{g}^{1}\kappa_{g}^{2},\\
\Delta_{2}=\left\vert
\begin{array}
[c]{cc}%
c_{11}^{1} & c_{22}^{1}\\
c_{11}^{2} & c_{22}^{2}%
\end{array}
\right\vert =-\kappa_{n}\left(
\begin{array}
[c]{l}%
{\small -3(\kappa}_{n}{\small )}^{3}{\small \tau}_{g}^{1}{\small \kappa}%
_{g}^{2}\left(  \kappa_{g}^{2}\right)  _{u}{\small +(\kappa}_{n}{\small )}%
^{5}\left(  \tau_{g}^{1}\right)  _{u}{\small +2(\kappa}_{n}{\small )}%
^{4}{\small \tau}_{g}^{1}\left(  \kappa_{n}\right)  _{u}\\
{\small -2\tau_{g}^{1}(}\left(  \kappa_{n}\right)  _{u}{\small )}%
^{3}{\small -(\kappa}_{n}{\small \kappa}_{g}^{2}{\small )}^{2}\left(
{\small 3\kappa}_{n}\left(  \tau_{g}^{1}\right)  _{u}{\small +5\tau}_{g}%
^{1}\left(  \kappa_{n}\right)  _{u}\right)  \\
{\small +2\kappa}_{n}\left(  \kappa_{n}\right)  _{u}\left(  \left(  \kappa
_{n}\right)  _{u}\left(  \tau_{g}^{1}\right)  _{u}{\small +2\tau}_{g}%
^{1}\left(  \kappa_{n}\right)  _{uu}\right)  {\small +(\kappa}_{n}%
{\small )}^{3}\left(  \tau_{g}^{1}\right)  _{uuu}\\
{\small +(\kappa}_{n}{\small )}^{2}\left(  {\small 5}\left(  \kappa
_{n}\right)  _{u}\left(  \tau_{g}^{1}\right)  _{uu}{\small +6}\left(  \tau
_{g}^{1}\right)  _{u}\left(  \kappa_{n}\right)  _{uu}{\small +2\tau}_{g}%
^{1}\left(  \kappa_{n}\right)  _{uuu}\right)
\end{array}
\right)  ,\\
\Delta_{3}=\left\vert
\begin{array}
[c]{cc}%
c_{12}^{1} & c_{22}^{1}\\
c_{12}^{2} & c_{22}^{2}%
\end{array}
\right\vert =(\kappa_{n})^{3}(\tau_{g}^{1})^{2}\kappa_{g}^{2}\left(
\begin{array}
[c]{l}%
{\small -(\kappa}_{n}{\small )}^{2}{\small \tau}_{g}^{1}{\small (\kappa}%
_{g}^{2}{\small )}^{2}{\small +2\tau_{g}^{1}(}\left(  \kappa_{n}\right)
_{u}{\small )}^{2}{\small +(\kappa}_{n}{\small )}^{2}\left(  \tau_{g}%
^{1}\right)  _{uu}\\
{\small +2\kappa}_{n}\left(  {\small 2}\left(  \kappa_{n}\right)  _{u}\left(
\tau_{g}^{1}\right)  _{u}{\small +\tau}_{g}^{1}\left(  \kappa_{n}\right)
_{uu}\right)
\end{array}
\right)  ,
\end{array}
\right\}  \label{delta123}%
\end{equation}
then we find the coefficients of the second fundamental form as%
\begin{equation}
\left.
\begin{array}
[c]{l}%
l_{11}=\frac{2\Delta_{1}}{\mathcal{W}}=2\kappa_{n}\kappa_{g}^{2},\\
l_{12}=\frac{\Delta_{2}}{\mathcal{W}}=\frac{-1}{\kappa_{n}\tau_{g}^{1}}\left(
\begin{array}
[c]{l}%
{\small -3(\kappa}_{n}{\small )}^{3}{\small \tau}_{g}^{1}{\small \kappa}%
_{g}^{2}\left(  \kappa_{g}^{2}\right)  _{u}{\small +(\kappa}_{n}{\small )}%
^{5}\left(  \tau_{g}^{1}\right)  _{u}{\small +2(\kappa}_{n}{\small )}%
^{4}{\small \tau}_{g}^{1}\left(  \kappa_{n}\right)  _{u}\\
{\small -2\tau_{g}^{1}(}\left(  \kappa_{n}\right)  _{u}{\small )}%
^{3}{\small -(\kappa}_{n}{\small \kappa}_{g}^{2}{\small )}^{2}\left(
{\small 3\kappa}_{n}\left(  \tau_{g}^{1}\right)  _{u}{\small +5\tau}_{g}%
^{1}\left(  \kappa_{n}\right)  _{u}\right)  \\
{\small +2\kappa}_{n}\left(  \kappa_{n}\right)  _{u}\left(  \left(  \kappa
_{n}\right)  _{u}\left(  \tau_{g}^{1}\right)  _{u}{\small +2\tau}_{g}%
^{1}\left(  \kappa_{n}\right)  _{uu}\right)  {\small +(\kappa}_{n}%
{\small )}^{3}\left(  \tau_{g}^{1}\right)  _{uuu}\\
{\small +(\kappa}_{n}{\small )}^{2}\left(  {\small 5}\left(  \kappa
_{n}\right)  _{u}\left(  \tau_{g}^{1}\right)  _{uu}{\small +6}\left(  \tau
_{g}^{1}\right)  _{u}\left(  \kappa_{n}\right)  _{uu}{\small +2\tau}_{g}%
^{1}\left(  \kappa_{n}\right)  _{uuu}\right)
\end{array}
\right)  ,\\
l_{22}=\frac{2\Delta_{3}}{\mathcal{W}}=2\kappa_{n}\tau_{g}^{1}\kappa_{g}%
^{2}\left(
\begin{array}
[c]{l}%
{\small -(\kappa}_{n}{\small )}^{2}{\small \tau}_{g}^{1}{\small (\kappa}%
_{g}^{2}{\small )}^{2}{\small +2\tau_{g}^{1}(}\left(  \kappa_{n}\right)
_{u}{\small )}^{2}{\small +(\kappa}_{n}{\small )}^{2}\left(  \tau_{g}%
^{1}\right)  _{uu}\\
{\small +2\kappa}_{n}\left(  {\small 2}\left(  \kappa_{n}\right)  _{u}\left(
\tau_{g}^{1}\right)  _{u}{\small +\tau}_{g}^{1}\left(  \kappa_{n}\right)
_{uu}\right)
\end{array}
\right)  .
\end{array}
\right\}  \label{lij}%
\end{equation}

Furthermore, if we consider the linear map
\[
\gamma:T_{p}S\longrightarrow T_{p}S
\]
which satisfies the conditions%
\[
\left.
\begin{array}
[c]{c}%
\gamma(\Omega_{u})=\gamma_{1}^{1}\Omega_{u}+\gamma_{1}^{2}\Omega_{v},\\
\gamma(\Omega_{v})=\gamma_{2}^{1}\Omega_{u}+\gamma_{2}^{2}\Omega_{v},
\end{array}
\right\}  \text{ \ \ }\left(  \gamma=\left[
\begin{array}
[c]{cc}%
\gamma_{1}^{1} & \gamma_{1}^{2}\\
\gamma_{2}^{1} & \gamma_{2}^{2}%
\end{array}
\right]  \right)  ,
\]
then we obtain that%
\begin{equation}
\left.
\begin{array}
[c]{l}%
\gamma_{1}^{1}=\frac{g_{12}l_{12}-g_{22}l_{11}}{g_{11}g_{22}-(g_{12})^{2}%
}=-2\kappa_{n}\kappa_{g}^{2},\\
\gamma_{1}^{2}=\frac{g_{12}l_{11}-g_{11}l_{12}}{g_{11}g_{22}-(g_{12})^{2}%
}=\frac{1}{(\kappa_{n})^{5}(\tau_{g}^{1})^{3}}\left(
\begin{array}
[c]{l}%
{\small -3(\kappa}_{n}{\small )}^{3}{\small \tau}_{g}^{1}{\small \kappa}%
_{g}^{2}\left(  \kappa_{g}^{2}\right)  _{u}{\small +(\kappa}_{n}{\small )}%
^{5}\left(  \tau_{g}^{1}\right)  _{u}{\small +2(\kappa}_{n}{\small )}%
^{4}{\small \tau}_{g}^{1}\left(  \kappa_{n}\right)  _{u}\\
{\small -2\tau_{g}^{1}(}\left(  \kappa_{n}\right)  _{u}{\small )}%
^{3}{\small -(\kappa}_{n}{\small \kappa}_{g}^{2}{\small )}^{2}\left(
{\small 3\kappa}_{n}\left(  \tau_{g}^{1}\right)  _{u}{\small +5\tau}_{g}%
^{1}\left(  \kappa_{n}\right)  _{u}\right) \\
{\small +2\kappa}_{n}\left(  \kappa_{n}\right)  _{u}\left(  \left(  \kappa
_{n}\right)  _{u}\left(  \tau_{g}^{1}\right)  _{u}{\small +2\tau}_{g}%
^{1}\left(  \kappa_{n}\right)  _{uu}\right)  {\small +(\kappa}_{n}%
{\small )}^{3}\left(  \tau_{g}^{1}\right)  _{uuu}\\
{\small +(\kappa}_{n}{\small )}^{2}\left(  {\small 5}\left(  \kappa
_{n}\right)  _{u}\left(  \tau_{g}^{1}\right)  _{uu}{\small +6}\left(  \tau
_{g}^{1}\right)  _{u}\left(  \kappa_{n}\right)  _{uu}{\small +2\tau}_{g}%
^{1}\left(  \kappa_{n}\right)  _{uuu}\right)
\end{array}
\right)  ,\\
\gamma_{2}^{1}=\frac{g_{12}l_{22}-g_{22}l_{12}}{g_{11}g_{22}-(g_{12})^{2}%
}=\frac{1}{\kappa_{n}\tau_{g}^{1}}\left(
\begin{array}
[c]{l}%
{\small -3(\kappa}_{n}{\small )}^{3}{\small \tau}_{g}^{1}{\small \kappa}%
_{g}^{2}\left(  \kappa_{g}^{2}\right)  _{u}{\small +(\kappa}_{n}{\small )}%
^{5}\left(  \tau_{g}^{1}\right)  _{u}{\small +2(\kappa}_{n}{\small )}%
^{4}{\small \tau}_{g}^{1}\left(  \kappa_{n}\right)  _{u}\\
{\small -2\tau_{g}^{1}(}\left(  \kappa_{n}\right)  _{u}{\small )}%
^{3}{\small -(\kappa}_{n}{\small \kappa}_{g}^{2}{\small )}^{2}\left(
{\small 3\kappa}_{n}\left(  \tau_{g}^{1}\right)  _{u}{\small +5\tau}_{g}%
^{1}\left(  \kappa_{n}\right)  _{u}\right) \\
{\small +2\kappa}_{n}\left(  \kappa_{n}\right)  _{u}\left(  \left(  \kappa
_{n}\right)  _{u}\left(  \tau_{g}^{1}\right)  _{u}{\small +2\tau}_{g}%
^{1}\left(  \kappa_{n}\right)  _{uu}\right)  {\small +(\kappa}_{n}%
{\small )}^{3}\left(  \tau_{g}^{1}\right)  _{uuu}\\
{\small +(\kappa}_{n}{\small )}^{2}\left(  {\small 5}\left(  \kappa
_{n}\right)  _{u}\left(  \tau_{g}^{1}\right)  _{uu}{\small +6}\left(  \tau
_{g}^{1}\right)  _{u}\left(  \kappa_{n}\right)  _{uu}{\small +2\tau}_{g}%
^{1}\left(  \kappa_{n}\right)  _{uuu}\right)
\end{array}
\right)  ,\\
\gamma_{2}^{2}=\frac{g_{12}l_{12}-g_{11}l_{22}}{g_{11}g_{22}-(g_{12})^{2}%
}=\frac{-2\kappa_{g}^{2}}{(\kappa_{n})^{3}\tau_{g}^{1}}\left(
\begin{array}
[c]{l}%
{\small -(\kappa}_{n}{\small )}^{2}{\small \tau}_{g}^{1}{\small (\kappa}%
_{g}^{2}{\small )}^{2}{\small +2\tau_{g}^{1}(}\left(  \kappa_{n}\right)
_{u}{\small )}^{2}{\small +(\kappa}_{n}{\small )}^{2}\left(  \tau_{g}%
^{1}\right)  _{uu}\\
{\small +2\kappa}_{n}\left(  {\small 2}\left(  \kappa_{n}\right)  _{u}\left(
\tau_{g}^{1}\right)  _{u}{\small +\tau}_{g}^{1}\left(  \kappa_{n}\right)
_{uu}\right)
\end{array}
\right)  .
\end{array}
\right\}  \label{gammaij}%
\end{equation}
Hence, we can give the following theorem:

\begin{theorem}
If $\Omega=\Omega(u,v)$ is a solution of the Betchov-Da Rios equation, then
\begin{align}
&  k(u,v)=\label{kuv}\\
&  \frac{-1}{(\kappa_{n})^{6}(\tau_{g}^{1})^{4}}\left(
\begin{array}
[c]{l}%
{\small 6(\kappa}_{n}{\small )}^{5}{\small \tau}_{g}^{1}{\small (\kappa}%
_{g}^{2}{\small )}^{3}{\small (\kappa}_{g}^{2}{\small )}_{u}\left(
3\kappa_{n}(\tau_{g}^{1})_{u}+5\tau_{g}^{1}(\kappa_{n})_{u}\right) \\
{\small +(\kappa}_{n}{\small \kappa}_{g}^{2}{\small )}^{4}\left(
{\small 4(\kappa}_{n}{\small )}^{2}{\small (\tau}_{g}^{1}{\small )}%
^{4}{\small +9(\kappa}_{n}(\tau_{g}^{1})_{u}{\small )}^{2}{\small +30\kappa
}_{n}{\small \tau}_{g}^{1}{\small (\kappa}_{n}{\small )}_{u}{\small (\tau}%
_{g}^{1}{\small )}_{u}{\small +25(\tau}_{g}^{1}(\kappa_{n})_{u}{\small )}%
^{2}\right) \\
{\small -6(\kappa}_{n}{\small )}^{3}{\small \tau}_{g}^{1}{\small \kappa}%
_{g}^{2}{\small (\kappa}_{g}^{2}{\small )}_{u}\left(
\begin{array}
[c]{l}%
{\small (\kappa}_{n}{\small )}^{5}{\small (\tau}_{g}^{1}{\small )}%
_{u}{\small +2(\kappa}_{n}{\small )}^{4}{\small \tau}_{g}^{1}{\small (\kappa
}_{n}{\small )}_{u}{\small -2\tau}_{g}^{1}\left(  (\kappa_{n})_{u}\right)
^{3}\\
{\small +2\kappa}_{n}{\small (\kappa}_{n}{\small )}_{u}\left(  (\kappa
_{n})_{u}(\tau_{g}^{1})_{u}+2\tau_{g}^{1}(\kappa_{n})_{uu}\right)
{\small +(\kappa}_{n}{\small )}^{3}{\small (\tau}_{g}^{1}{\small )}_{uuu}\\
{\small +(\kappa}_{n}{\small )}^{2}\left(  {\small 5(\kappa}_{n}{\small )}%
_{u}{\small (\tau}_{g}^{1}{\small )}_{uu}{\small +6(\tau_{g}^{1})_{u}(\kappa
}_{n}{\small )}_{uu}{\small +2\tau}_{g}^{1}{\small (\kappa}_{n}{\small )}%
_{uuu}\right)
\end{array}
\right) \\
{\small +}\left(
\begin{array}
[c]{l}%
{\small (\kappa}_{n}{\small )}^{5}{\small (\tau}_{g}^{1}{\small )}%
_{u}{\small +2(\kappa}_{n}{\small )}^{4}{\small \tau}_{g}^{1}{\small (\kappa
}_{n}{\small )}_{u}{\small -2\tau}_{g}^{1}\left(  (\kappa_{n})_{u}\right)
^{3}\\
{\small +2\kappa}_{n}{\small (\kappa}_{n}{\small )}_{u}\left(  (\kappa
_{n})_{u}(\tau_{g}^{1})_{u}+2\tau_{g}^{1}(\kappa_{n})_{uu}\right)
{\small +(\kappa}_{n}{\small )}^{3}{\small (\tau}_{g}^{1}{\small )}_{uuu}\\
{\small +(\kappa}_{n}{\small )}^{2}\left(  {\small 5(\kappa}_{n}{\small )}%
_{u}{\small (\tau}_{g}^{1}{\small )}_{uu}{\small +6(\tau_{g}^{1})_{u}(\kappa
}_{n}{\small )}_{uu}{\small +2\tau}_{g}^{1}{\small (\kappa}_{n}{\small )}%
_{uuu}\right)
\end{array}
\right)  ^{2}\\
{\small -(\kappa}_{n}{\small \kappa}_{g}^{2}{\small )}^{2}\left(
\begin{array}
[c]{l}%
{\small 4(\kappa}_{n}{\small \tau}_{g}^{1}{\small )}^{3}\left(  4(\kappa
_{n})_{u}(\tau_{g}^{1})_{u}+\kappa_{n}(\tau_{g}^{1})_{uu}\right) \\
{\small +8(\kappa}_{n}{\small (\tau}_{g}^{1}{\small )}^{2}{\small )}%
^{2}\left(  \left(  (\kappa_{n})_{u}\right)  ^{2}+\kappa_{n}(\kappa_{n}%
)_{uu}\right) \\
{\small +6(\kappa}_{n}{\small )}^{2}{\small (\tau}_{g}^{1}{\small )}%
_{u}\left(
\begin{array}
[c]{l}%
{\small (\kappa}_{n}{\small )}^{4}{\small (\tau}_{g}^{1}{\small )}%
_{u}{\small +2}\left(  (\kappa_{n})_{u}\right)  ^{2}{\small (\tau}_{g}%
^{1}{\small )}_{u}{\small +(\kappa}_{n}{\small )}^{2}{\small (\tau}_{g}%
^{1}{\small )}_{uuu}\\
{\small +\kappa}_{n}\left(  {\small 5(\kappa}_{n}{\small )}_{u}{\small (\tau
}_{g}^{1}{\small )}_{uu}{\small +6(\tau_{g}^{1})_{u}(\kappa}_{n}%
{\small )}_{uu}\right)
\end{array}
\right) \\
{\small +(\tau}_{g}^{1}{\small )}^{2}\left(
\begin{array}
[c]{l}%
{\small -20}\left(  (\kappa_{n})_{u}\right)  ^{4}{\small -(\kappa}%
_{n}{\small )}^{4}\left(  {\small 9}\left(  (\kappa_{g}^{2})_{u}\right)
^{2}{\small -20}\left(  (\kappa_{n})_{u}\right)  ^{2}\right) \\
{\small +40\kappa}_{n}\left(  (\kappa_{n})_{u}\right)  ^{2}{\small (\kappa
}_{n}{\small )}_{uu}{\small +20(\kappa}_{n}{\small )}^{2}{\small (\kappa}%
_{n}{\small )}_{u}{\small (\kappa}_{n}{\small )}_{uuu}%
\end{array}
\right) \\
{\small +2\kappa}_{n}{\small \tau}_{g}^{1}\left(
\begin{array}
[c]{l}%
{\small 11(\kappa}_{n}{\small )}^{4}{\small (\kappa}_{n}{\small )}%
_{u}{\small (\tau}_{g}^{1}{\small )}_{u}{\small +4}\left(  (\kappa_{n}%
)_{u}\right)  ^{3}{\small (\tau}_{g}^{1}{\small )}_{u}\\
{\small +\kappa}_{n}{\small (\kappa}_{n}{\small )}_{u}\left(  25(\kappa
_{n})_{u}(\tau_{g}^{1})_{uu}+42(\tau_{g}^{1})_{u}(\kappa_{n})_{uu}\right) \\
{\small +(\kappa}_{n}{\small )}^{2}\left(  5(\kappa_{n})_{u}(\tau_{g}%
^{1})_{uuu}+6(\tau_{g}^{1})_{u}(\kappa_{n})_{uuu}\right)
\end{array}
\right)
\end{array}
\right)
\end{array}
\right) \nonumber
\end{align}
and%
\begin{equation}
h(u,v)=\frac{\kappa_{g}^{2}\left(  \left(  -(\kappa_{n}\kappa_{g}^{2}%
)^{2}+(\kappa_{n})^{4}+2(\left(  \kappa_{n}\right)  _{u})^{2}+2\kappa
_{n}\left(  \kappa_{n}\right)  _{uu}\right)  \tau_{g}^{1}+\left(  4\left(
\kappa_{n}\right)  _{u}\left(  \tau_{g}^{1}\right)  _{u}+\kappa_{n}\left(
\tau_{g}^{1}\right)  _{uu}\right)  \kappa_{n}\right)  }{(\kappa_{n})^{3}%
\tau_{g}^{1}} \label{huv}%
\end{equation}
are the invariants of the soliton surface $S:\Omega=\Omega(u,v)$ with respect
to the ED$^{2}$-frame field in $E^{4}$.
\end{theorem}

\begin{proof}
From (\ref{gammaij}),%
\[
k(u,v)=\det(\gamma(u,v))
\]
and%
\[
h(u,v)=-\frac{1}{2}tr(\gamma(u,v)),
\]
we obtain the invariants as (\ref{kuv}) and (\ref{huv}).
\end{proof}

On the other hand, from (\ref{gij}), (\ref{W}) and (\ref{cijk}), we find the
coefficients of the shape operator matrices according to the orthonormal
normal frame field $\{E,N\}$ of $S$ as%
\begin{equation}
\left.
\begin{array}
[c]{l}%
h_{11}^{1}=\frac{c_{11}^{1}}{g_{11}}=0,\\
h_{11}^{2}=\frac{c_{11}^{2}}{g_{11}}=\kappa_{n},\\
h_{12}^{1}=\frac{1}{\mathcal{W}}\left(  c_{12}^{1}-\frac{g_{12}}{g_{11}}%
c_{11}^{1}\right)  =-\kappa_{g}^{2},\\
h_{12}^{2}=\frac{1}{\mathcal{W}}\left(  c_{12}^{2}-\frac{g_{12}}{g_{11}}%
c_{11}^{2}\right)  =0,\\
h_{22}^{1}=\frac{1}{\mathcal{W}^{2}}\left(  g_{11}c_{22}^{1}-2g_{12}c_{12}%
^{1}+\frac{(g_{12})^{2}}{g_{11}}c_{11}^{1}\right)  \\
\text{ \ \ \ \ }=\frac{1}{\left(  (\kappa_{n})^{2}\tau_{g}^{1}\right)  ^{2}%
}\left(
\begin{array}
[c]{l}%
-3(\kappa_{n})^{3}\tau_{g}^{1}\kappa_{g}^{2}\left(  \kappa_{g}^{2}\right)
_{u}+(\kappa_{n})^{5}\left(  \tau_{g}^{1}\right)  _{u}+2(\kappa_{n})^{4}%
\tau_{g}^{1}\left(  \kappa_{n}\right)  _{u}\\
-2\tau_{g}^{1}(\left(  \kappa_{n}\right)  _{u})^{3}-(\kappa_{n}\kappa_{g}%
^{2})^{2}\left(  3\kappa_{n}\left(  \tau_{g}^{1}\right)  _{u}+5\tau_{g}%
^{1}\left(  \kappa_{n}\right)  _{u}\right)  \\
+2\kappa_{n}\left(  \kappa_{n}\right)  _{u}\left(  \left(  \kappa_{n}\right)
_{u}\left(  \tau_{g}^{1}\right)  _{u}+2\tau_{g}^{1}\left(  \kappa_{n}\right)
_{uu}\right)  +(\kappa_{n})^{3}\left(  \tau_{g}^{1}\right)  _{uuu}\\
+(\kappa_{n})^{2}\left(  5\left(  \kappa_{n}\right)  _{u}\left(  \tau_{g}%
^{1}\right)  _{uu}+6\left(  \tau_{g}^{1}\right)  _{u}\left(  \kappa
_{n}\right)  _{uu}+2\tau_{g}^{1}\left(  \kappa_{n}\right)  _{uuu}\right)
\end{array}
\right)  ,\\
h_{22}^{2}=\frac{1}{\mathcal{W}^{2}}\left(  g_{11}c_{22}^{2}-2g_{12}c_{12}%
^{2}+\frac{(g_{12})^{2}}{g_{11}}c_{11}^{2}\right)  \\
\text{ \ \ \ \ }=\frac{1}{(\kappa_{n})^{2}\tau_{g}^{1}}\left(  \kappa_{n}%
\tau_{g}^{1}\left(  \kappa_{g}^{2}\right)  ^{2}-4\left(  \kappa_{n}\right)
_{u}\left(  \tau_{g}^{1}\right)  _{u}-\kappa_{n}\left(  \tau_{g}^{1}\right)
_{uu}-\frac{2\tau_{g}^{1}}{\kappa_{n}}\left(  (\left(  \kappa_{n}\right)
_{u})^{2}+\kappa_{n}\left(  \kappa_{n}\right)  _{uu}\right)  \right)  .
\end{array}
\right\}  \label{hij}%
\end{equation}
From (\ref{hij}), we find the shape operator matrices according to normal
vector fields $E$ and $N$ of $S$ as%
\begin{equation}
A_{E}=\left[
\begin{array}
[c]{cc}%
h_{11}^{1} & h_{12}^{1}\\
h_{12}^{1} & h_{22}^{1}%
\end{array}
\right]  \text{ and }A_{N}=\left[
\begin{array}
[c]{cc}%
h_{11}^{2} & h_{12}^{2}\\
h_{12}^{2} & h_{22}^{2}%
\end{array}
\right]  .\label{AF12}%
\end{equation}
Now, we can obtain the Gaussian curvature, mean curvature vector field and
Gaussian torsion of the soliton surface $S$. Also, we can give some important
geometric characterizations such as minimal, flat and semi-umbilic soliton
surfaces with respect to the ED$^{2}$-frame field in $E^{4}$.

\subsection{Flat \textbf{Betchov-Da Rios soliton surface with respect to the
ED}$^{2}$\textbf{-frame field in }$E^{4}$}

\

Here, first, we will give the following theorem which states the Gaussian
curvature of the soliton surface $S$ in order to provide a characterization of
the flatness of this surface.

\begin{theorem}
If $\Omega=\Omega(u,v)$ is a solution of the Betchov-Da Rios equation with
respect to the ED$^{2}$-frame field in $E^{4}$, then the Gaussian curvature of
the soliton surface $S:\Omega=\Omega(u,v)$ is%
\begin{equation}
K=-\frac{\kappa_{n}\left(  4\left(  \kappa_{n}\right)  _{u}\left(  \tau
_{g}^{1}\right)  _{u}+\kappa_{n}\left(  \tau_{g}^{1}\right)  _{uu}\right)
+2\tau_{g}^{1}\left(  \left(  \left(  \kappa_{n}\right)  _{u}\right)
^{2}+\kappa_{n}\left(  \kappa_{n}\right)  _{uu}\right)  }{(\kappa_{n})^{2}%
\tau_{g}^{1}}. \label{Gaussian}%
\end{equation}

\end{theorem}

\begin{proof}
Using (\ref{hij}) and (\ref{AF12}), we obtain the Gaussian curvature of $S$
from%
\[
K=\det(A_{E})+\det(A_{N}).
\]
\end{proof}

So, we have

\begin{theorem}
Let $\Omega=\Omega(u,v)$ be a solution of the Betchov-Da Rios equation with
respect to the ED$^{2}$-frame field in $E^{4}$. The soliton surface
$S:\Omega=\Omega(u,v)$ is flat if and only if $a_{34}=\kappa_{n}\tau_{g}%
^{1}\left(  \kappa_{g}^{2}\right)  ^{2}$ holds, where $a_{34}$ is given by
(\ref{a34}).
\end{theorem}

\begin{proof}
From (\ref{a34}) and (\ref{Gaussian}), we can write the Gaussian curvature as%
\begin{equation}
K=\frac{a_{34}-\kappa_{n}\tau_{g}^{1}\left(  \kappa_{g}^{2}\right)  ^{2}%
}{\kappa_{n}\tau_{g}^{1}} \label{Ky}%
\end{equation}
and this completes the proof.
\end{proof}

\subsection{Minimal \textbf{Betchov-Da Rios soliton surface with respect to
the ED}$^{2}$\textbf{-frame field in }$E^{4}$}

\

Now, let us present the following theorem which states the mean curvature
vector field of the soliton surface. This theorem serves to provide a
characterization of the surface's minimality.

\begin{theorem}
If $\Omega=\Omega(u,v)$ is a solution of the Betchov-Da Rios equation with
respect to the ED$^{2}$-frame field in $E^{4}$, then the mean curvature vector
field of the soliton surface $S:\Omega=\Omega(u,v)$ is%
\begin{equation}
\vec{H}=\frac{1}{2\left(  (\kappa_{n})^{2}\tau_{g}^{1}\right)  ^{2}}\left(
\begin{array}
[c]{c}%
\left(
\begin{array}
[c]{l}%
-3(\kappa_{n})^{3}\tau_{g}^{1}\kappa_{g}^{2}\left(  \kappa_{g}^{2}\right)
_{u}+(\kappa_{n})^{5}\left(  \tau_{g}^{1}\right)  _{u}+2(\kappa_{n})^{4}%
\tau_{g}^{1}\left(  \kappa_{n}\right)  _{u}\\
-2\tau_{g}^{1}(\left(  \kappa_{n}\right)  _{u})^{3}-(\kappa_{n}\kappa_{g}%
^{2})^{2}\left(  3\kappa_{n}\left(  \tau_{g}^{1}\right)  _{u}+5\tau_{g}%
^{1}\left(  \kappa_{n}\right)  _{u}\right) \\
+2\kappa_{n}\left(  \kappa_{n}\right)  _{u}\left(  \left(  \kappa_{n}\right)
_{u}\left(  \tau_{g}^{1}\right)  _{u}+2\tau_{g}^{1}\left(  \kappa_{n}\right)
_{uu}\right)  +(\kappa_{n})^{3}\left(  \tau_{g}^{1}\right)  _{uuu}\\
+(\kappa_{n})^{2}\left(  5\left(  \kappa_{n}\right)  _{u}\left(  \tau_{g}%
^{1}\right)  _{uu}+6\left(  \tau_{g}^{1}\right)  _{u}\left(  \kappa
_{n}\right)  _{uu}+2\tau_{g}^{1}\left(  \kappa_{n}\right)  _{uuu}\right)
\end{array}
\right)  E\\
+\kappa_{n}\tau_{g}^{1}\left(
\begin{array}
[c]{l}%
(\kappa_{n}\kappa_{g}^{2})^{2}\tau_{g}^{1}-\kappa_{n}\left(  4\left(
\kappa_{n}\right)  _{u}\left(  \tau_{g}^{1}\right)  _{u}+\kappa_{n}\left(
\tau_{g}^{1}\right)  _{uu}\right) \\
+\tau_{g}^{1}\left(  (\kappa_{n})^{4}-2\left(  \left(  \kappa_{n}\right)
_{u}\right)  ^{2}-2\kappa_{n}\left(  \kappa_{n}\right)  _{uu}\right)
\end{array}
\right)  N
\end{array}
\right)  . \label{mean}%
\end{equation}

\end{theorem}

\begin{proof}
Using (\ref{hij}) and (\ref{AF12}), we obtain the mean curvature vector field
of $S$ from%
\[
\vec{H}=\frac{1}{2}\left(  tr(A_{E})E+tr(A_{N})N\right)  .
\]
\end{proof}

Thus,

\begin{theorem}
Let $\Omega=\Omega(u,v)$ be a solution of the Betchov-Da Rios equation with
respect to the ED$^{2}$-frame field in $E^{4}$. The soliton surface
$S:\Omega=\Omega(u,v)$ is minimal if and only if $a_{23}=0$ and $a_{34}%
=-\left(  \kappa_{n}\right)  ^{3}\tau_{g}^{1}$ hold, where $a_{23}$ and
$a_{34}$ are given by (\ref{a23}) and (\ref{a34}), respectively.
\end{theorem}

\begin{proof}
From (\ref{a34}), (\ref{a23}) and (\ref{mean}), the mean curvature vector
field can be written as%
\begin{equation}
\vec{H}=\frac{\left(  -a_{23}(\kappa_{n})^{2}\tau_{g}^{1}\right)  E+\left(
\kappa_{n}\tau_{g}^{1}\left(  a_{34}\kappa_{n}+(\kappa_{n})^{4}\tau_{g}%
^{1}\right)  \right)  N}{2\left(  (\kappa_{n})^{2}\tau_{g}^{1}\right)  ^{2}}
\label{Hy}%
\end{equation}
and so, the proof completes.
\end{proof}

\subsection{Semi-umbilic \textbf{Betchov-Da Rios soliton surface with respect
to the ED}$^{2}$\textbf{-frame field in }$E^{4}$}

\

In this subsection, we will obtain the Gaussian torsion of the soliton surface
$S$ and state a theorem which contains the necessary and sufficient conditions
for semi-umbilic soliton surface.

The Gaussian torsion (also called the normal curvature function) of a surface
$M\subset E^{4}$ given by a regular patch $\Psi(u,v)$ is (\cite{Aminov},
\cite{Desmet}, \cite{Guadalup}, \cite{Gutierrez})
\begin{equation}
K_{N}=\frac{g_{11}\left(  c_{12}^{1}c_{22}^{2}-c_{12}^{2}c_{22}^{1}\right)
-g_{12}\left(  c_{11}^{1}c_{22}^{2}-c_{11}^{2}c_{22}^{1}\right)
+g_{22}\left(  c_{11}^{1}c_{12}^{2}-c_{11}^{2}c_{12}^{1}\right)  }%
{\mathcal{W}^{3}}. \label{KN}%
\end{equation}
So, by using (\ref{gij}), (\ref{W}) and (\ref{cijk}) in (\ref{KN}), we get

\begin{theorem}
If $\Omega=\Omega(u,v)$ is a solution of the Betchov-Da Rios equation with
respect to the ED$^{2}$-frame field in $E^{4}$, then the Gaussian torsion of
the soliton surface $S:\Omega=\Omega(u,v)$ is%
\begin{equation}
K_{N}=\frac{\kappa_{g}^{2}}{(\kappa_{n})^{3}\tau_{g}^{1}}\left(  \left(
-(\kappa_{n}\kappa_{g}^{2})^{2}+(\kappa_{n})^{4}+2\left(  \left(  \kappa
_{n}\right)  _{u}\right)  ^{2}+2\kappa_{n}\left(  \kappa_{n}\right)
_{uu}\right)  \tau_{g}^{1}+\left(  4\left(  \kappa_{n}\right)  _{u}\left(
\tau_{g}^{1}\right)  _{u}+\kappa_{n}\left(  \tau_{g}^{1}\right)  _{uu}\right)
\kappa_{n}\right)  . \label{KNy}%
\end{equation}

\end{theorem}

A point $p\in M$ is semi-umbilic if and only if $K_{N}(p)=0$ and a surface $M$
immersed in $E^{4}$ is said to be semi-umbilical provided all its points are
semi-umbilic \cite{Gutierrez}. Hence, from (\ref{a34}) and (\ref{KNy}), we can
prove the following theorem:

\begin{theorem}
Let $\Omega=\Omega(u,v)$ be a solution of the Betchov-Da Rios equation with
respect to the ED$^{2}$-frame field in $E^{4}$. The soliton surface
$S:\Omega=\Omega(u,v)$ is semi-umbilic if and only if $a_{34}\kappa_{g}%
^{2}=\left(  \kappa_{n}\right)  ^{3}\tau_{g}^{1}\kappa_{g}^{2}$ holds, where
$a_{34}$ is given by (\ref{a34}).
\end{theorem}

\begin{proof}
From (\ref{a34}) and (\ref{KNy}), the Gaussian torsion can be written as%
\begin{equation}
K_{N}=\frac{(\kappa_{n})^{3}\tau_{g}^{1}\kappa_{g}^{2}-a_{34}\kappa_{g}^{2}%
}{(\kappa_{n})^{2}\tau_{g}^{1}} \label{KNyy}%
\end{equation}
and so, the proof completes.
\end{proof}

\section{\textbf{Curvature Ellipse of the Betchov-Da Rios soliton surface with
respect to the ED}$^{2}$\textbf{-frame field in }$E^{4}$}

Let $M\subset E^{4}$ be a surface and let us take a circle given by
the angle $\theta\in\lbrack0,2\pi]$ in the tangent space $T_{p}M$ at
a point $p\in M$. Thus, let the intersection curve of $M$ and the
hyperplane, which is the direct sum of the normal plane
$T_{p}^{\bot}M$ and the line generated by the direction vector
$X=\cos\theta X_{1}+\sin\theta X_{2}$ at the point $p\in M$, be
denoted by $\gamma_{\theta}$. Here, the vectors $X_{1}$ and $X_{2}$
are orthonormal basis of $T_{p}M$. This curve is called the
\textit{normal section curve} at point $p$ of $M$ and in the
direction $X$. Also, the normal curvature vector $\eta_{\theta}$ of
$\gamma_{\theta}$ is a vector lying in $T_{p}^{\bot}M$ and when the
angle $\theta$ varies from $0$ to $2\pi$, this vector constructs an
ellipse in $T_{p}^{\bot}M$. This ellipse is called a
\textit{curvature ellipse} at a point $p$ of $M$. Now, if the vector
\[
\gamma_{\theta}^{\prime}=X=\cos\theta X_{1}+\sin\theta X_{2}%
\]
is a unit vector of the normal section curve, then the normal curvature
ellipse $\left\Vert \eta_{\theta}\right\Vert $ is given by%
\begin{align*}
x  &  =\frac{h_{11}^{1}+h_{22}^{1}}{2}+\frac{h_{11}^{1}-h_{22}^{1}}{2}%
\cos2\theta+h_{12}^{1}\sin2\theta,\\
y  &  =\frac{h_{11}^{2}+h_{22}^{2}}{2}+\frac{h_{11}^{2}-h_{22}^{2}}{2}%
\cos2\theta+h_{12}^{2}\sin2\theta.
\end{align*}
Thus, the curvature ellipse at the point $p$ of $M$ is denoted by%
\[
E(p)=\{h(X,X):X\in T_{p}M,\text{ }\left\Vert X\right\Vert =1\}\text{,}%
\]
where $h$ is the second fundamental form of the patch $X(u,v)$. To see that
this indicates an ellipse, with the aid of $X=\cos\theta X_{1}+\sin\theta
X_{2}$, it is enough to examine the formula%
\[
h(X,X)=\vec{H}+\cos2\theta\vec{B}+\sin2\theta\vec{C},
\]
where $\vec{H}$ is the mean curvature vector and $\vec{B}$, $\vec{C}$ are
normal vectors given by%
\[
\vec{B}=\frac{h(X_{1},X_{1})-h(X_{2},X_{2})}{2},\text{ \ }\vec{C}%
=h(X_{1},X_{2})\text{.}%
\]
This shows us that when $X$ makes one revolution around the unit circle, the
vector $h(X,X)$ makes two revolutions around the ellipse centered at $\vec{H}%
$. This ellipse is the $E(p)$ ellipse of $X(u,v)$ at the point $p$. Clearly,
the ellipse $E(p)$ can degenerate to a point or a line. For more details about
the curvature ellipse of surfaces, we refer to \cite{Little}, \cite{Mochida},
\cite{Rouxel}, \cite{Wong}, and etc.

Now, let us recall the following invariants that characterize the curvature
ellipse of surfaces.

The determinant $\Delta(p)$ and matrix $A(p)$ for a surface $M\subset E^{4},$
given by a regular patch $M:\Psi(u,v),$ are defined with the aid of
(\ref{hij}) by%
\begin{equation}
\Delta(p)=\frac{1}{4}\det\left[
\begin{array}
[c]{cccc}%
h_{11}^{1} & 2h_{12}^{1} & h_{22}^{1} & 0\\
h_{11}^{2} & 2h_{12}^{2} & h_{22}^{2} & 0\\
0 & h_{11}^{1} & 2h_{12}^{1} & h_{22}^{1}\\
0 & h_{11}^{2} & 2h_{12}^{2} & h_{22}^{2}%
\end{array}
\right]  (p) \label{Delta}%
\end{equation}
and%
\begin{equation}
A(p)=\left[
\begin{array}
[c]{ccc}%
h_{11}^{1} & h_{12}^{1} & h_{22}^{1}\\
h_{11}^{2} & h_{12}^{2} & h_{22}^{2}%
\end{array}
\right]  (p), \label{A(p)}%
\end{equation}
respectively. With the aid of these invariants, one can give the following
classifications for the origin $p$ of the normal space $T_{p}^{\bot}M$:

a) If $\Delta(p)<0$, then the point $p$ lies outside the curvature ellipse and
such a point is called a hyperbolic point of $M$.

b) If $\Delta(p)>0$, then the point $p$ lies inside the curvature ellipse and
such a point is called an elliptic point of $M$.

c) If $\Delta(p)=0$, then the point $p$ lies on the curvature ellipse and such
a point is called a parabolic point of $M$. For this case, we have the
following detailed possibilities:

\ \ i) If $\Delta(p)=0$ and $K(p)>0$, then the point $p$ is an inflection
point of imaginary type.

\ \ ii) If $\Delta(p)=0$, $K(p)<0$ and $rank(A(p))=2$, then the ellipse is
non-degenerate; if $\Delta(p)=0$, $K(p)<0$ and $rank(A(p))=1$, then the point
$p$ is an inflection point of real type.

\ \ iii) If $\Delta(p)=0$ and $K(p)=0$, then the point $p$ is an inflection
point of flat type \cite{Mochida}.

By using (\ref{hij}) in (\ref{Delta}) and (\ref{A(p)}), from the above
definitions, we have

\begin{theorem}
Let $S:\Omega=\Omega(u,v)$ be a solution of the Betchov-Da Rios equation with
respect to the ED$^{2}$-frame field in $E^{4}$. Then the origin $p$ of the
normal space $T_{p}^{\bot}S$ can be classified by the following cases:

\begin{description}
\item[a] If the inequality $4a_{34}\kappa_{n}\tau_{g}^{1}\left(  \kappa
_{g}^{2}\right)  ^{2}+(a_{23})^{2}>0$ is satisfied, then $p$ lies outside the
curvature ellipse and so, it is a hyperbolic point of $S$.

\item[b] If the inequality $4a_{34}\kappa_{n}\tau_{g}^{1}\left(  \kappa
_{g}^{2}\right)  ^{2}+(a_{23})^{2}<0$ is satisfied, then $p$ lies inside the
curvature ellipse and so, it is an elliptic point of $S$.

\item[c] If the conditions $4a_{34}\kappa_{n}\tau_{g}^{1}\left(  \kappa
_{g}^{2}\right)  ^{2}+(a_{23})^{2}=0$ and $\kappa_{g}^{2}\neq0$ are satisfied,
then $p$ is non-degenerate. Also, let the conditions $4a_{34}\kappa_{n}%
\tau_{g}^{1}\left(  \kappa_{g}^{2}\right)  ^{2}+(a_{23})^{2}=0$ and
$\kappa_{g}^{2}=0$ are satisfied. In this case;

if $a_{34}\ $and $\kappa_{n}\tau_{g}^{1}$ have the same signs, then $p$ is an
inflection point of imaginary type;

if $a_{34}\ $and $\kappa_{n}\tau_{g}^{1}$ have the opposite signs, then $p$ is non-degenerate;

if $a_{34}=0$ holds, then $p$ is an inflection point of flat type.
\end{description}
\end{theorem}

\begin{proof}
From (\ref{hij}) and (\ref{Delta}), we obtain the invariant $\Delta(p)$ as%
\begin{equation}
\Delta(p)=\frac{1}{4\left(  \kappa_{n}\right)  ^{7}\left(  \tau_{g}%
^{1}\right)  ^{4}}\left(
\begin{array}
[c]{l}%
-4\left(  \kappa_{n}\right)  ^{5}\left(  \tau_{g}^{1}\right)  ^{3}\left(
\kappa_{g}^{2}\right)  ^{2}\left(
\begin{array}
[c]{l}%
\left(  \kappa_{n}\tau_{g}^{1}(\kappa_{g}^{2})^{2}-4\left(  \kappa_{n}\right)
_{u}\left(  \tau_{g}^{1}\right)  _{u}-\kappa_{n}\left(  \tau_{g}^{1}\right)
_{uu}\right)  \kappa_{n}\\
-2\left(  \left(  \left(  \kappa_{n}\right)  _{u}\right)  ^{2}+\kappa
_{n}\left(  \kappa_{n}\right)  _{uu}\right)  \tau_{g}^{1}%
\end{array}
\right)  \\
-\kappa_{n}\left(
\begin{array}
[c]{l}%
-3(\kappa_{n})^{3}\tau_{g}^{1}\kappa_{g}^{2}\left(  \kappa_{g}^{2}\right)
_{u}+(\kappa_{n})^{5}\left(  \tau_{g}^{1}\right)  _{u}+2(\kappa_{n})^{4}%
\tau_{g}^{1}\left(  \kappa_{n}\right)  _{u}-2\tau_{g}^{1}\left(  \left(
\kappa_{n}\right)  _{u}\right)  ^{3}\\
-\left(  \kappa_{n}\kappa_{g}^{2}\right)  ^{2}\left(  3\kappa_{n}\left(
\tau_{g}^{1}\right)  _{u}+5\tau_{g}^{1}\left(  \kappa_{n}\right)  _{u}\right)
\\
+2\kappa_{n}\left(  \kappa_{n}\right)  _{u}\left(  \left(  \kappa_{n}\right)
_{u}\left(  \tau_{g}^{1}\right)  _{u}+2\tau_{g}^{1}\left(  \kappa_{n}\right)
_{uu}\right)  \\
+(\kappa_{n})^{3}\left(  \tau_{g}^{1}\right)  _{uuu}+(\kappa_{n})^{2}\left(
\begin{array}
[c]{l}%
5\left(  \kappa_{n}\right)  _{u}\left(  \tau_{g}^{1}\right)  _{uu}+6\left(
\tau_{g}^{1}\right)  _{u}\left(  \kappa_{n}\right)  _{uu}\\
+2\tau_{g}^{1}\left(  \kappa_{n}\right)  _{uuu}%
\end{array}
\right)
\end{array}
\right)  ^{2}%
\end{array}
\right)  .\label{Deltay}%
\end{equation}
By using (\ref{a34}) and (\ref{a23}) in (\ref{Deltay}), we get%
\begin{equation}
\Delta(p)=-\frac{4a_{34}\kappa_{n}\tau_{g}^{1}\left(  \kappa_{g}^{2}\right)
^{2}+(a_{23})^{2}}{4\left(  \kappa_{n}\tau_{g}^{1}\right)  ^{2}}%
.\label{Deltayy}%
\end{equation}
Firstly, if $4a_{34}\kappa_{n}\tau_{g}^{1}\left(  \kappa_{g}^{2}\right)
^{2}+(a_{23})^{2}>0$ is satisfied, then we have $\Delta(p)<0$ and so (a) is proved.

Secondly, if $4a_{34}\kappa_{n}\tau_{g}^{1}\left(  \kappa_{g}^{2}\right)
^{2}+(a_{23})^{2}<0$ is satisfied, then we have $\Delta(p)>0$ and so (b) is proved.

Finally, let us assume that the equation $4a_{34}\kappa_{n}\tau_{g}^{1}\left(
\kappa_{g}^{2}\right)  ^{2}+(a_{23})^{2}=0$ is satisfied.

If $\kappa_{g}^{2}\neq0$ (at the beginning of the second section, we stated
that $\kappa_{n}\tau_{g}^{1}\neq0$) in this case, then we obtain the Gaussian
curvature as $K=-\frac{4\left(  \kappa_{n}\tau_{g}^{1}\right)  ^{2}\left(
\kappa_{g}^{2}\right)  ^{4}+(a_{23})^{2}}{4\left(  \kappa_{n}\tau_{g}%
^{1}\kappa_{g}^{2}\right)  ^{2}}$ and so, this is always negative. From
(\ref{hij}) and (\ref{A(p)}), the rank of matrix $A(p)$ is $2$. So, the first
part of (c) is completed.

If $\kappa_{g}^{2}=0$ in this case, then we obtain the Gaussian curvature as
$K=\frac{a_{34}}{\kappa_{n}\tau_{g}^{1}}$ and so, the cases of second part of
(c) are proved.
\end{proof}

\section{\textbf{Wintgen Ideal Betchov-Da Rios soliton surface with respect to
the ED}$^{2}$\textbf{-frame field in }$E^{4}$}

In this section, we prove a theorem that characterizes the Wintgen ideal
(superconformal) Betchov-Da Rios soliton surface with respect to the ED$^{2}%
$-frame field in $E^{4}$.

An important inequality%
\[
K+\left\vert K_{N}\right\vert \leq\left\Vert \vec{H}\right\Vert ^{2}%
\]
for Gaussian curvature $K$, mean curvature vector field $\vec{H}$ and Gaussian
torsion $K_{N}$ of a surface in $E^{4}$ has been proved by Wintgen in 1979
\cite{Wintgen}. Also the equality, i.e.
\begin{equation}
K+\left\vert K_{N}\right\vert =\left\Vert \vec{H}\right\Vert ^{2}
\label{wintgen}%
\end{equation}
holds if and only if the curvature ellipse is a circle.

A surface in $E^{4}$ is called a Wintgen ideal (superconformal) surface if it
satisfies the equation (\ref{wintgen}).

So, we can give the following theorem which states the necessary conditions
for a Betchov-Da Rios soliton surface to be Wintgen ideal with respect to the
ED$^{2}$-frame field in $E^{4}$:

\begin{theorem}
Let $\Omega=\Omega(u,v)$ be a solution of the Betchov-Da Rios equation with
respect to the ED$^{2}$-frame field in $E^{4}$. The soliton surface
$S:\Omega=\Omega(u,v)$ is Wintgen ideal (superconformal) if and only if
$a_{23}=0$ and $a_{34}=\left(  \kappa_{n}\right)  ^{2}\tau_{g}^{1}\left(
\kappa_{n}-2\kappa_{g}^{2}\right)  $ hold, where $a_{23}$ and $a_{34}$ are
given by (\ref{a23}) and (\ref{a34}), respectively.
\end{theorem}

\begin{proof}
If we use (\ref{Ky}), (\ref{Hy}) and (\ref{KNy}) in (\ref{wintgen}), then we
get%
\[
\frac{(a_{23})^{2}+\left(  a_{34}+\left(  \kappa_{n}\right)  ^{2}\tau_{g}%
^{1}\left(  2\kappa_{g}^{2}-\kappa_{n}\right)  \right)  ^{2}}{4\left(
(\kappa_{n})^{2}\tau_{g}^{1}\right)  ^{2}}=0
\]
and this completes the proof.
\end{proof}

\section{\textbf{An Application for the Betchov-Da Rios Soliton Surface}}

In this section, we construct a soliton surface $\Omega(u,v)$
associated with the Betchov-Da Rios equation and find the
ED$^{2}$-frame field of the $u$-parameter curve $\Omega(u,v)$ for
all $v$ in $E^{4}$. Additionally, we obtain its geometric invariants
$k$ and $h$, the Gaussian curvature $K$, the mean curvature vector
field $\vec{H}$ and Gaussian torsion $K_{N}$. To better understand
our example, we can visualize it by projecting the soliton surface
into 3-dimensional spaces.

Let us consider the soliton surface as%
\begin{equation}
\Omega(u,v)=\left(  \frac{\cos u-u}{2},\frac{\cos u+u}{2},\frac{\sin u}%
{\sqrt{2}},\frac{v}{2\sqrt{2}}\right)  . \label{exsurf}%
\end{equation}
Here, the $u$-parameter curves $\Omega(u,v)$ (for all $v$) of the soliton
surface (\ref{exsurf}) are lying on the hypersurface $M:f(x,y,z,w)=(x+y)^{2}%
+2z^{2}-1=0$ and also, one can check that (\ref{exsurf}) satisfies the
Betchov-Da Rios equation in $E^{4}$. In the following figure, one can see the
projection of the hypersurface $M$ and the $u$-parameter curve $\Omega(u,v)$
for all $v$ into $xyz$-space.

\begin{figure}[H]
\centering
\includegraphics[
height=2.5in, width=6.1in
]{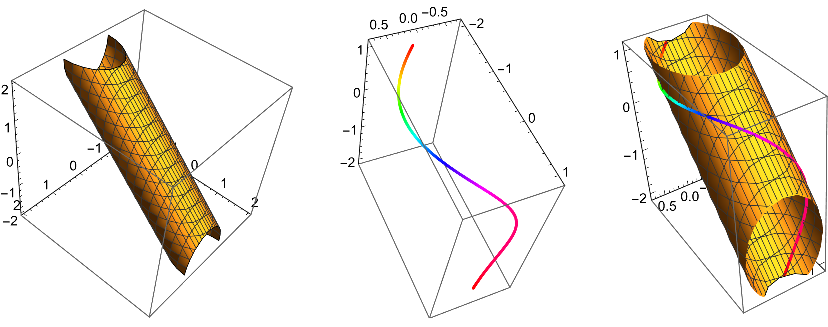}\caption{ }%
\label{fig:1}%
\end{figure}

The unit tangent vector of the $u$-parameter curve $\Omega=\Omega(u,v)$ for
all $v$ on the Betchov-Da Rios soliton surface (\ref{exsurf}) is%
\begin{equation}
T(u,v)=\left(  \frac{-1-\sin u}{2},\frac{1-\sin u}{2},\frac{\cos u}{\sqrt{2}%
},0\right)  .\label{exF1}%
\end{equation}
Furthermore, the unit normal vector field of the hypersurface $M$ is%
\[
\mathcal{N}=\frac{1}{\sqrt{2}}(x+y,x+y,2z,0)
\]
and thus, for all $v\in%
\mathbb{R}
$ we get
\begin{equation}
N(u,v)=\mathcal{N}(\Omega(u,v))=\left(  \frac{\cos u}{\sqrt{2}},\frac{\cos
u}{\sqrt{2}},\sin u,0\right)  .\label{exF4}%
\end{equation}
Because of $\left(  \Omega(u,v)\right)  _{uu}$ (for all $v\in%
\mathbb{R}
$) is linear dependent with $N(u,v),$ we can obtain the remaining frame
vectors of the ED$^{2}$-frame field along the curves of $\Omega(u,v)$ for all
$v\in%
\mathbb{R}
$ as following:%
\begin{equation}
E(u,v)=\left(  \frac{-1+\sin u}{2},\frac{1+\sin u}{2},\frac{-\cos u}{\sqrt{2}%
},0\right)  \label{exF2}%
\end{equation}
and%
\begin{equation}
D(u,v)=(0,0,0,1).\label{exF3}%
\end{equation}
Also, the normal curvature, geodesic curvature of order 2 and geodesic torsion
of order 1 are obtained by%
\begin{equation}
\kappa_{n}(u,v)=-\frac{1}{\sqrt{2}},\text{ }\kappa_{g}^{2}(u,v)=0,\text{
\ }\tau_{g}^{1}(u,v)=\frac{1}{\sqrt{2}}.\label{excurvatures}%
\end{equation}
On the other hand, we obtain the geometric invariants $k$, $h$ and the
Gaussian curvature, mean curvature vector field and Gaussian torsion of the
soliton surface (\ref{exsurf}) as%
\begin{equation}
k=h=K=K_{N}=0,\text{ }\vec{H}=-\frac{1}{2\sqrt{2}}N\text{.}\label{exsurfcurv}%
\end{equation}
Since we find that $k=h=0,$ we reach that the soliton surface consists of flat points.

Also, the determinant $\Delta(p)$ of the soliton surface (\ref{exsurf}) is%
\begin{equation}
\Delta(p)=0\label{Deltaex}%
\end{equation}
for all points $p$ and so, all points of the soliton surface are inflection
points of flat type.

Finally, let us present the figures of the Betchov-Da Rios soliton surface
(\ref{exsurf})\ projections into $xyz,$ $xyw,$ $xzw$ and $yzw$-spaces. These
projections are shown in figures (a), (b), (c), and (d) respectively.

\begin{figure}[H]
\centering
\includegraphics[
height=2.1in, width=6.9in
]{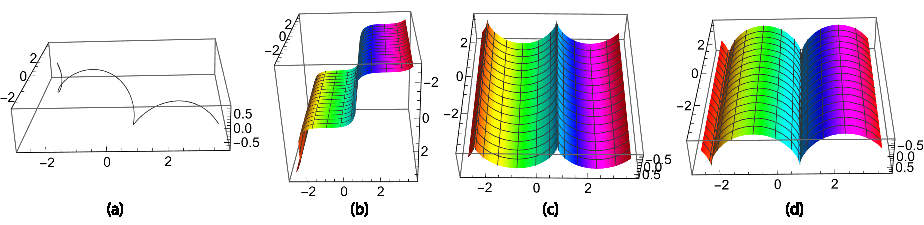}\caption{Projections of the Betchov-Da Rios soliton surface
(\ref{exsurf})}%
\label{fig:2}%
\end{figure}

\section{\textbf{Conclusion and Future Work}}

In this study, for a soliton surface $\Omega=\Omega(u,v)$ associated with the
Betchov-Da Rios equation, firstly we give the derivative formulas of ED$^{2}%
$-frame field of a unit speed curve $u$-parameter curve $\Omega=\Omega(u,v)$
for all $v$. After that, we obtain two geometric invariants $k$ and $h$ of the
soliton surface and we obtain the Gaussian curvature, mean curvature vector
and Gaussian torsion of $\Omega$. With the aid of these surface invariants, we
give some theorems which contain the conditions for flat, minimal and
semi-umbilic soliton surfaces. Also, by obtaining the determinant $\Delta(p)$
and matrix $A(p)$ for a soliton surface, we give an important theorem which
contains the curvature ellipse of the Betchov-Da Rios soliton surface with
respect to ED$^{2}$-frame field in $E^{4}$. We prove a theorem which
characterizes the Wintgen ideal (superconformal) Betchov-Da Rios soliton
surface with respect to the ED$^{2}$-frame field in $E^{4}$. Finally, we
construct an example for Betchov-Da Rios soliton surface with the aid of the
ED$^{2}$-frame field in $E^{4}$, find its geometric invariants and give its
visualizations into 3-space.

We hope that this study will give a new perspective to readers who deal with
the geometric properties of the Betchov-Da Rios equation. As open problems,
the Betchov-Da Rios soliton surface with the aid of different frame fields in
four-dimensional Euclidean space or Minkowski spacetime can give us important
results. Also, maybe interesting results can be obtained by using the visco-Da
Rios equation instead of the Betchov-Da Rios equation.

\end{document}